\documentclass[11pt]{article}
\usepackage{amscd}
\usepackage{amsmath}
\usepackage{latexsym}
\usepackage{amsfonts}
\usepackage{amssymb}
\usepackage{amsthm}
\usepackage{graphicx}
\usepackage{verbatim}
\usepackage{enumerate}
\usepackage{xcolor}

\theoremstyle{plain}
\theoremstyle{definition}\newtheorem{theorem}{Theorem}[section]
\theoremstyle{plain}\newtheorem{lemma}[theorem]{Lemma}
\theoremstyle{plain}\newtheorem{coro}[theorem]{Corollary}
\theoremstyle{plain}\newtheorem{proposition}[theorem]{Proposition}
\theoremstyle{remark}\newtheorem{remark}{Remark}[section]
\theoremstyle{definition}
\theoremstyle{definition}\newtheorem*{definition1}{Definition}

\newcommand{\Div}{\mathrm{div}}

\newcommand{\be}{\begin{equation}}
\newcommand{\ee}{\end{equation}}
 \newcommand{\ba}{\begin{aligned}}
 \newcommand{\ea}{\end{aligned}}
  \newcommand{\f}{\frac}
    
  \newcommand{\ben}{\begin{enumerate}}
   \newcommand{\een}{\end{enumerate}}
\newcommand{\te}{\text}

\newcommand{\Rmnum}[1]{\expandafter\@slowromancap\romannumeral #1@}

\numberwithin{equation}{section}
\begin{document}
\title{Anisotropic regularity conditions for the suitable weak solutions to the 3d Navier-Stokes equations}
\author{Yanqing Wang\footnote{School of Mathematical Sciences, Capital Normal University, Beijing 100048, PR China. Email: wangyanqing20056@gmail.com.}\;  and  Gang Wu\footnote{School of Mathematical Sciences, University of Chinese Academy of Sciences, Beijing 100049, PR China. Email: wugangmaths@gmail.com.}}
\date{}
\maketitle

\begin{abstract}
We are concerned with the problem,
originated from   Seregin \cite{[Seregin3],[Seregin1],[Seregin2]},
what are minimal sufficiently conditions for the regularity
 of suitable weak solutions to the 3d Naiver-Stokes equations.
We prove  some interior regularity criteria,
    in terms of either one  component of the velocity
    with sufficiently small local scaled norm and the rest  part
     with bounded local scaled norm,  or horizontal part of the vorticity  with sufficiently small local scaled norm
 and the vertical  part with bounded local scaled norm.
It is  also shown that only the smallness on the local scaled $L^{2}$ norm of horizontal gradient without any other condition
on the vertical gradient can still ensure the regularity of suitable weak solutions.
All these conclusions improve
  pervious results on the local scaled norm type
 regularity conditions.
 \end{abstract}

\noindent {\bf MSC(2000):} 35Q30, 35A02. \\
{\bf Keywords:} Navier-Stokes equations, suitable weak solutions, regularity.

\section{Introduction}

In this paper we consider the following classical incompressible 3d Navier-Stokes equations
\begin{align}
&u_{t}- \Delta u+u\cdot \nabla u+\nabla \Pi=0,~~\Div\,u=0,~(x,t)\in\Omega\times(0,\,T),\, u|_{\partial\Omega\times (0,T)}=0,\label{NS}
\end{align}
where $\Omega\subseteq\mathbb{R}^{3}$ is a bounded regular domain,
the vector field $u$ denotes  velocity  of the flow, the
scalar function $\Pi$ stands for pressure of the fluid.
The  initial data  $u(0)$ is also divergence-free.

In 1970s,  Scheffer \cite{[Scheffer1],[Scheffer2],[Scheffer4]} introduced the
concept of the suitable weak solutions, later developed by Caffarelli,
  Kohn,   Nirenberg \cite{[CKN]} and Lin \cite{[Lin]},
to the 3d Naiver-Stokes equations.
In contrast to the usual Leray-Hopf  weak solutions
equipped with the energy inequality,
the suitable weak solutions
  enjoy the  following local energy inequality  (inverse-Sobolev inequality)
\be\label{loc}
\begin{aligned}
&\int_{\Omega}|u(t',x)|^2\phi dx+2\int^{t'}_{t}\int_{\Omega}|\nabla u(s,x)|^2\phi dxds\\
\leq&\int^{t'}_{t}\int_{\Omega}|u(s,x)|^2
\Big(\partial_{t}\phi+\Delta\phi\Big)dx ds
+\int^{t'}_{t}\int_{\Omega}(|u(s,x)|^2 +2|\Pi(s,x)|)u(s,x)\cdot\nabla\phi dx ds,
         \end{aligned}
\ee
for any non-negative function $\phi\in C_{0}^{\infty}(\Omega\times(t,t'))$.\\
We now recall the definition of suitable weak solutions for the Navier-Stokes equations.
  \begin{definition1}[Suitable weak solutions] A pair $(u, \,\Pi)$ is said to be a suitable weak solution to the Navier-Stokes
equations $(1.1)$, provided the following conditions are satisfied
\begin{enumerate}[(i)]
 \item$u\in L^{\infty}(t,t';L^{2}(\Omega))\cap L^{2}(t,t';W^{1,2}(\Omega)),$~
 $\Pi\in L^{5/3}(t,t';L^{5/3}(\Omega)).$
\item$(u,\Pi)$ solves (1.1) in $\Omega\times (t,t') $ in the sense of distributions.
\item$(u,\Pi)$ satisfies the  local energy inequality \eqref{loc}. \end{enumerate}
 \end{definition1}

For the sake of statements, we denote by $\omega=\text{curl}\,u$ the vorticity of the flow,
and by $u_{h}=(u_{1},u_{2},0)$
 and $\omega_{h}=(\omega_{1},\omega_{2},0)$
the horizontal part of the velocity and
 the vorticity, respectively. Similarly, the horizontal gradient operator
 and vertical  gradient operator are denoted by
 $\nabla_{h}=(\partial_{1},\partial_{2},0)$
 and $\nabla_{3}=(0,0,\partial_{3})$, respectively.
 Throughout this paper, we also set
$$\ba
&B(x,r)=\{y\in \mathbb{R}^{3}||x-y|\leq r\},~~ &B&(r)= B(0,r),\\
&Q(x,t,r)=B(x,r)\times(t-r^{2},t),~~ &Q&(r)= Q(0,0,r).
\ea$$

A point is said to be a regular point to \eqref{NS} if $u$ is
 bounded in some neighborhood
of this point. The rest points will be called singular points.
 Making full use of the local energy inequality \eqref{loc},
 Scheffer \cite{[Scheffer1],[Scheffer2],[Scheffer4]} could estimate the size of
 the potential space-time
singular points set of suitable weak solutions.
The
optimal estimate of the Hausdorff dimension of the possible
 singular points set was obtained by Caffarelli,   Kohn,
  Nirenberg \cite{[CKN]}
 via establishing  the following regularity criterion:
 there is an absolute constant $\varepsilon_{1}$ such that, if
   \be\label{ckn}
  \limsup_{r\rightarrow0}\f{1}{r}\iint_{Q(r)}|\nabla u|^{2}dxdt
  \leq \varepsilon_{1},
  \ee
  then $(0,0)$ is a regular point.

Since then, there has been much effort to study the partial
    regularity of suitable weak solutions, see, for example,
\cite{[CKL],[DG],[GKT],[LS],[Lin],[NP],
[SS],[Seregin3],[Seregin1],[Seregin2],[TX],[Vasseur],[WZ],[WW]}.
 On   one hand, various kinds of alternative approach to
  the partial regularity theory of 3d Navier-stokes equations
  are developed. By means of  blow-up procedure and compact method,
 Lin \cite{[Lin]} gave a  simple proof  of   Caffarelli-Kohn-Nirenberg's theorem. See also Ladyzenskaja and Seregin \cite{[LS]} for more details.
 Applying the De Giorgi iteration argument
to the  Navier-Stokes equations,  Vasseur \cite{[Vasseur]} rebuilt the
  main results in \cite{[CKN]}.
 It should be pointed out that  the
 final task of all these proof is to
 prove that  \eqref{ckn} holds. Very recently, an analogue of    Caffarelli-Kohn-Nirenberg's theorem in \cite{[CKN]}
  for the suitable weak solutions to the  4d Navier-Stokes equations is   proved
 in \cite{[WW]}, where
 an analogous regularity condition \eqref{ckn} is verified, see also Dong and Gu \cite{[DG]}.
However, it is not known whether the Caffarelli-Kohn-Nirenberg type regularity criterion \eqref{ckn} is valid
for the higher dimensional Naiver-Stokes equations.
 On the other hand,   sufficient  regularity conditions  on   local scaled norm  similar to \eqref{ckn} are generalized and improved,
see, for example, \cite{[GKT],[SS],[Seregin3],[Seregin1],[Seregin2],[TX],[WZ]}.
Particularly, Tian and Xin    \cite{[TX]} established the following criteria:
there exists an absolute constant $\varepsilon_{2}$ such that, if the suitable weak solution $u$ satisfies
   \be\label{tx}
  \limsup_{r\rightarrow0}\f{1}{r}\iint_{Q(r)}|\omega|^{2}dxdt\leq \varepsilon_{2} ~~
  \text{or }~~ \limsup_{r\rightarrow0}\f{1}{r^{2}}\iint_{Q(r)}| u|^{3}dxdt\leq \varepsilon_{2}
  ,\ee
  then $(0,0)$ is a regular point. Later, the regularity conditions \eqref{ckn} and \eqref{tx}
 are further generalized and strengthened by
  Gustafson,   Kang and  Tsai in    \cite{[GKT]} to
\begin{equation}\label{tsai1}
\limsup_{r\to 0 }\,\, r^{1-  \frac 3p -\frac 2q}
\Big(\int^{0}_{-r^{2}}\Big(\int_{B(r)}|u|^{p}dx\Big)^{\f{q}{p}}ds\Big)^{\f{1}{q}} \leq \varepsilon_{3},\quad
1\leq \f{3}{p} +\f{2}{q} \leq 2,\;  1\leq p, q \leq \infty;
\end{equation}
or
\begin{equation}
\limsup_{r\to 0 }\,\, r^{2-  \frac 3p -\frac 2q}
\Big(\int^{0}_{-r^{2}}\Big(\int_{B(r)}|\nabla u|^{p }dx\Big)^{\f{q}{p }}ds\Big)^{\f{1}{q}} \leq \varepsilon_{3},\quad
2\le \f{3}{p} +\f{2}{q} \le 3, \; 1 \le p,q \le \infty;
\end{equation}
or
\begin{equation}\label{gkt2}
\begin{split}\limsup_{r\to 0 }\,\, r^{2-  \frac 3p -\frac 2q}
\Big(\int^{0}_{-r^{2}}\Big(\int_{B(r)}|\omega|^{p }dx\Big)^{\f{q}{p }}ds\Big)^{\f{1}{q}}
 \leq \varepsilon_{3}, \\
 2\le \f{3}{p} +\f{2}{q} \le 3, \; 1 \le p,q \le \infty,\;(p, q)\neq(1, \infty);
 \end{split}
 \end{equation}
 where $\varepsilon_{3}$ is an absolute constant. For other versions
   of the local scaled norm type
 regularity conditions, we refer the reader to \cite{[SS],[Seregin3]} and references therein.

  In \cite{[Seregin3],[Seregin1],[Seregin2]},
 Seregin began to address the problem what are minimal conditions
 which guarantee the regularity of suitable weak solutions.
 To be more precisely,
some  results  in  \cite{[Seregin2]} read as follows.  For any $M>0$, there exists a positive number $\varepsilon_4(M)$ such that, if
 $$
 \limsup_{r\rightarrow0}\f{1}{r}\iint_{Q(r)}|\nabla u|^{2}dxdt\leq M
 ~\text{ and}~~
  \liminf_{r\rightarrow0}\f{1}{r}\iint_{Q(r)}|\nabla_{3} u|^{2}dxdt\leq \varepsilon_{4}(M),
  $$
 or
 $$\
 \limsup_{r\rightarrow0}\f{1}{r}\iint_{Q(r)}|  u|^{3}dxdt\leq M
 ~\text{ and}~~
  \liminf_{r\rightarrow0}\f{1}{r}\iint_{Q(r)}| u|^{3}dxdt\leq \varepsilon_{4}(M),
  $$
then $(0, 0)$ is a regular point. This improved  the famous Caffarelli-Kohn-Nirenberg¡¯s condition.  The proof is mainly  based on used a blow-up 	
procedure.
Utilizing this method,
 Wang and Zhang \cite{[WZ]} recently showed that the
 smallness of horizontal part of the velocity
 is enough to ensure the regularity of suitable weak solutions, namely,
 for any constant $M>0$, there is a positive number $\varepsilon_5(M)$ such that
 $u$ is regular at $(0,0)$ if one of the following conditions holds,
\begin{equation}\label{wz1}\ba
&\ &\limsup_{r\to 0}\,\, r^{1-  \frac 3p -\frac 2q}
\Big(\int^{0}_{-r^{2}}\Big(\int_{B(r)}|u|^{p}dx\Big)
^{\f{q}{p}}ds\Big)^{\f{1}{q}} \leq M
 \\
&\text{and} &
\liminf_{r\to 0}\,\, r^{1-\frac 3{p } - \frac 2q  }
\Big(\int^{0}_{-r^{2}}\Big(\int_{B(r)}|u_{h}|^{p}dx
\Big)^{\f{q}{p}}ds\Big)^{\f{1}{q}}\leq \varepsilon_{5}(M),
\ea\end{equation} $\text{with}~
  1\leq \f{3}{p} +\f{2}{q} <2, \,  1< p,\, q \leq \infty;$
\begin{equation}\label{wz2}\ba
 &\ &\limsup_{r\to 0}\,\, r^{2-  \frac 3p -\frac 2q}
\Big(\int^{0}_{-r^{2}}\Big(\int_{B(r)}|\nabla u|^{p}dx\Big)^{\f{q}{p}}ds\Big)^{\f{1}{q}}
\leq M  \\
&\text{and} &\liminf_{r\to 0}\,\, r^{-( \frac 3p + \frac 2q -2)}
\Big(\int^{0}_{-r^{2}}\Big(\int_{B(r)}|\nabla u_{h}|^{p}dx\Big)^{\f{q}{p}}ds\Big)^{\f{1}{q}}
\leq \varepsilon_{5}(M), \ea\end{equation}
  with $2\leq\f{3}{p} +\f{2}{q} < 3,
 \, 1 < p,\,q \le \infty.$

The main objective of this paper is to
  refine the above regularity criteria for the
   regularity of the suitable weak solutions.
 Notice that
 the results obtained by Seregin \cite{[Seregin2]}  and Wang and Zhang \cite{[WZ]} do not  include the regularity criterion in terms of
 the vorticity which is one of  most important physical quantities in  fluid flows (see, for example, \cite{[BKM],[MB]}).
Our first result
 is to establish a regular condition involving  the smallness on the horizontal part of the vorticity.
Secondly,  we show that the smallness  of
  horizontal part  of the velocity
  in \eqref{wz1} and \eqref{wz2} can be reduced to the
   smallness of only one component of the velocity.
 It is worth noting that all previous
 sufficient  conditions in  \cite{[GKT],[Seregin2],[TX],[WZ]}
  involve full
 components of the velocity,
  vorticity or gradient of the velocity rather than
  genuine partial  components.
Our third result is to show that
the smallness of  horizontal gradient without  any condition
on the vertical gradient can still imply the
regularity of suitable weak solutions.
\begin{theorem}\label{the2}
Suppose that the pair $(u, \Pi)$ be a suitable weak solution to \eqref{NS} in $Q(1)$. For any
constant $M>0$, there exists a positive constant $\varepsilon_{11}(M)$ such that if $\omega \in L^{p}L^{q}(Q(1)),$
 \be\label{the2c}
\limsup_{r\rightarrow0} r^{2-  \frac 3{p } - \frac {2}{q  }}
\Big(\int^{0}_{-r^{2}}\Big(\int_{B(r)}|\omega|^{p }dx\Big)^{\f{q}{p }}ds\Big)^{\f{1}{q }}
 \leq M,
 \ee
\text{and}\be
 \liminf_{r\rightarrow0} r^{2-  \frac 3{p } - \frac {2}{q  }}
\Big(\int^{0}_{-r^{2}}\Big(\int_{B(r)}|\omega_{h}|^{p }dx\Big)^{\f{q}{p }}ds\Big)^{\f{1}{q }}
 \leq\varepsilon_{11}(M),\ee
with $2\leq3/p+2/q<3$, $1<p,\,q\leq\infty $,
then $(0,\,0)$ is a regular point.
\end{theorem}
This result is an extension of criterion \eqref{gkt2}. The proof relies on the blow-up method  used in \cite{[Seregin2],[WZ]}.
 The key point is how to gain the interior
regularity for the suitable weak solution to the following system
  \be\label{spe2}\left\{\ba
&u_{t}-\Delta u+u\cdot\nabla u+\nabla \Pi=0,\\
&\text{div}\, u=0,~~\omega_{h}=0.\\
\ea\right.\ee
As a matter of fact, if two components of a vorticity belongs
 to  $L_{x}^{p}L_{t}^{q}$ for $3/p+2/q\leq2$,
 then the suitable weak solution of the Navier-Stokes equations
  is regular, which is shown by Chae,
   Kang and  Lee in \cite{[CKL]}. To make our paper more self-contained and more readable, we shall outline    an alternative path to the proof of interior
regularity for the system \eqref{spe2} in Appendix.
We use elementary tools involving the  special structure of equations \eqref{spe2}
and the
  bootstrapping argument utilized by Serrin
in \cite{[Serrin]}.
We would like to note that
 we work on the velocity equations rather than the
  vorticity equations used  in \cite{[Serrin]}.
    This together with the   Biot-Savart
 law and the Calder\'on-Zygmund Theorem allows us
  to prove Theorem \ref{the2}.

\begin{theorem}\label{the1}
Let $(u,\, \Pi)$ be a suitable weak solutions to \eqref{NS} in $Q(1)$.
For any
constant $M>0$, there exists a positive constant $\varepsilon_{21}(M)$(or $\varepsilon_{22}(M)$)
such that $(0,0)$ is regular point if
\begin{enumerate}[(1)]
\item \label{item1} $u\in L^{p}L^{q}(Q(1))$, with
  $1\leq 3/p+2/q <2, 1<p,\,q\leq\infty,$ \begin{align}\label{the1c}
  \limsup_{r\to 0}\,\, r^{1- \frac 3{p } - \frac 2q  }
\Big(\int^{0}_{-r^{2}}\Big(\int_{B(r)}|u|^{p}dx\Big)^{\f{q}{p}}ds\Big)^{\f{1}{q}}
  \leq M, \end{align}
 \text{and}  \begin{align}
 \liminf_{r\to 0}\,\, r^{1- \frac 3{p } - \frac {2}{q}}
\Big(\int^{0}_{-r^{2}}\Big(\int_{B(r)}|u_{3}|^{p}dx\Big)^{\f{q}{p}}ds\Big)
^{\f{1}{q}}\leq\varepsilon_{21}(M);\nonumber
\end{align}
\item $\nabla u\in L^{p}L^{q}(Q(1))$, with $2\leq3/p+2/q < 3, \,1<p,\,q\leq\infty,$ \begin{align}\label{the1asum}
 \limsup_{r\to 0}\,\, r^{2-  \frac 3p -\frac 2q  }
\Big(\int^{0}_{-r^{2}}\Big(\int_{B(r)}|\nabla u|^{p}
dx\Big)^{\f{q}{p}}ds\Big)^{\f{1}{q}}
 \leq M, \end{align}
and
 \begin{align}
 \liminf_{r\to 0}\,\, r^{2-  \frac 3{p }
  - \frac {2}{q}}
\Big(\int^{0}_{-r^{2}}\Big(\int_{B(r)}
|\nabla u_{3}|^{p}dx\Big)^{\f{q}{p}}ds\Big)
^{\f{1}{q}} \leq \varepsilon_{22}(M).
\end{align}
  \end{enumerate}
\end{theorem}
Notice that any  suitable weak solution to the following system
\be\label{spe1}\left\{\ba
&u_{t}-\Delta u+u\cdot\nabla u+\nabla \Pi=0,\\
&\text{div}\, u=0,~~u_{3}=0,\\
\ea\right.\ee
is bounded,
which was proved by Neustupa and  Penel in \cite{[NP]}.
The result of \cite{[NP]} says that
if any one component of the velocity, which
is a suitable weak solutions to the 3d Navier-Stokes equations, is bounded, then the velocity
has no singular point.
Then the proof will be done in the same spirit as Theorem \ref{the2}.

With  the first regular criterion  in Theorem \ref{the1} in hand, following the path of
 Section 4 in \cite{[WZ]},  one can further extend the  Ladyzhenskaya-Prodi-Serrin's criterion obtained there. We leave this to
 the interesting reader.



Now we turn to the third result in this paper.
Note that the following version of Sobolev and Ladyzhenskaya's inequalities
\be\label{ls}
\|f\|_{L^{p}(\mathbb{R}^{3})}\leq C \|f\|^{\f{6-p}{2p}}_{L^{2}(\mathbb{R}^{3})}
\|\partial_{1}f\|^{\f{p-2}{2p}}_{L^{2}(\mathbb{R}^{3})}
\|\partial_{2}f\|^{\f{p-2}{2p}}_{L^{2}(\mathbb{R}^{3})}
\|\partial_{3}f\|^{\f{p-2}{2p}}_{L^{2}(\mathbb{R}^{3})},
 \text{with} ~~2\leq p\leq 6,
\ee
are frequently used in the investigation of
models of incompressible
fluid mechanics. For instance,
see \cite{[CW],[CDGE],[MZ],[Zheng]} and references therein.
Here, we are able to prove the  local  version of these inequalities
 (see Lemma \ref{keylemma} for details), which entails
  the following theorem.
\begin{theorem}\label{the3}
There exists an absolute constant $\varepsilon_{31}$
with the following property. If $(u,\Pi)$ is a suitable weak solution and
 \be\label{the3c}
  \limsup_{r\rightarrow0}\f{1}{r}\iint_{Q(r)}|\nabla_{h} u|^{2}dxds
  \leq \varepsilon_{31},
  \ee
  then (0,0) is a regular point.
\end{theorem}
This is an
improvement of condition \eqref{ckn}
 given by   Caffarelli,  Kohn and
  Nirenberg in \cite{[CKN]}.
As a by-product of Theorem \ref{the3},   H\"older's  inequality and absolute continuity of Lebesgue's integral immediately yield that
\begin{coro}
 Suppose that $(u,\Pi)$ is a suitable weak solution to \eqref{NS}. If there exist a constant $r$ such that
 $$\nabla_{h}u\in L^{p}_{x}L^{q}_{t}(Q(r)),
 ~\text{with}~\f{3}{p}+\f{2}{q}\leq2, \,\,2\leq p,\,q\leq\infty,$$
 then $(0,0)$ is regular point.
\end{coro}

This  paper is organized as follows. In the  second section,
we introduce some notations, recall some known conclusions and prove some auxiliary lemmas which will play a key role in the proof of our results.
 Section 3 contains the proofs of Theorem 1.1 and 1.2.
  Section 4 is devoted to prove  Theorem 1.3.
Finally,  an appendix
 is dedicated to  the proof of the regularity of the solution to system \eqref{spe2}.

\section{Notations and some auxiliary lemmas} \label{section2}

 For $p\in [1,\,\infty]$, the notation $L^{p}((0,\,T);X)$ stands for the set of measurable functions on the interval $(0,\,T)$ with values in $X$ and $\|f(t,\cdot)\|_{X}$ belongs to $L^{p}(0,\,T)$.  For simplicity,   we write $\|f\| _{L^{p}L^{q}(Q(r))}:=\|f\| _{L^{q}(-r^{2},0;L^{p}(B(r)))} $ and
 $\|f\| _{L^{p}(Q(r))}:=\|f\| _{L^{p}L^{p}(Q(r))}$.
 We will also use  the summation convention on repeated indices.
 $C$ is an absolute
   constant which may be different from line to line unless otherwise stated.

Due to the
scaling property of Navier-Stokes equations \cite{[CKN]},
we introduce the following dimensionless quantities
\begin{align}
&E_{p}(u,\,r)=\frac{1}{r^{5-p}}\iint_{Q(r)}|u|^{p}dx dt,
&&E_{\ast}(u,\,r)=\frac{1}{r }\iint_{Q(r)}|\nabla u|^2dxdt,  \nonumber\\
&E(u,\,r)=\sup_{-r^2\leq   t<0}\frac{1}{r }\int_{B(r)}|u|^2dx,
&&P_{q}(\Pi,r)=\frac{1}{r^{5-2q}}\iint_{Q(r)}|\Pi|^{q}dx
dt, \nonumber\\
&E_{\ast,h}(u,\,r)=\frac{1}{r }\iint_{Q(r)}|\nabla_{h} u|^2dxdt,
&&E_{\ast,3}(u,\,r)=\frac{1}{r }\iint_{Q(r)}|\nabla_{3} u|^2dxdt
, \nonumber\\
&W(\omega,\,r)=\frac{1}{r }\iint_{Q(r)}|\omega|^2dxdt,
&&W_{h}(\omega,\,r)=\frac{1}{r }\iint_{Q(r)}|\omega_{h}|^2dxdt
, \nonumber \end{align}
\begin{align}
&E_{p,q}(u,\,r)=r^{1-\f{3}{p}-\f{2}{q}}\Big(\int^{0}_{-r^{2}}\Big(\int_{B(r)}|u|^{p}dx\Big)
^{\f{q}{p}}dt\Big)^{\f{1}{q}},
\nonumber\\
&E_{\ast;p,q}(u,\,r)=r^{2-\f{3}{p}-\f{2}{q}}\Big(\int^{0}_{-r^{2}}\Big(\int_{B(r)}|\nabla u|^{p }dx\Big)^{\f{q}{p }}dt\Big)^{\f{1}{q}},\nonumber\\
&W_{p,q}(\omega,\,r)=r^{2-\f{3}{p}-\f{2}{q}}\Big(\int^{0}_{-r^{2}}\Big(\int_{B(r)}|\omega|^{p }dx\Big)^{\f{q}{p }}dt\Big)^{\f{1}{q}}.\nonumber
\end{align}

First, we recall some known results.
\begin{lemma}\label{lem1}\cite{[Seregin1],[WZ]}
Let $(u,\Pi)$ be a suitable weak solution to (\ref{NS}) in $Q(1)$. Assume
that, for any constant $M>0$ and $r\in(0,1]$, there hold
\be \ba
&&r^{1-\f{3}{p}-\f{2}{q}}
\Big(\int^{0}_{-r^{2}}\Big(\int_{B(r)}|u|^{p}
dx\Big)^{\f{q}{p}}ds\Big)^{\f{1}{q}}\leq
 M\quad\text{ with }\quad1\le\frac3p+\frac2q<2,
  1<q\leq \infty
 \ea\ee
 or
 \be \ba
&&r^{2-\f{3}{p}-\f{2}{q}}\Big(\int^{0}_{-r^{2}}
\Big(\int_{B(r)}|\nabla u|^{p }dx\Big)^{\f{q}{p }}ds
\Big)^{\f{1}{q}}\leq M\quad\text{ with }
\quad 2\le \frac3{p}+\frac2{q}<3, 1<p\le\infty.
\ea\ee
Then,   for any $r\in (0,\f{1}{2}]$,
\be
E(u,r)+E_{\ast}(u,r)+P_{3/2}(\Pi,r)\leq C(p,q,M)\big(r^{1/2}\big(E_{3}(u,1)+P_{3/2}(\Pi,1)\big)+1\big).
\ee
\end{lemma}
\begin{theorem}\label{Lin}\cite{[LS],[Lin]}
Let $(u,\Pi)$ be a suitable weak solution to (\ref{NS}) in $Q(1) $. There exists
$\varepsilon_0>0$ such that if
\be
\iint_{Q(1)}|u|^3+|\Pi|^{3/2}dxdt\leq \varepsilon_0,
\ee
then $u$ is regular in $Q(1/2)$.
\end{theorem}
Secondly, we
present a decay type estimate of  the pressure. For the other versions
of decay estimate for the pressure, we refer the reader to \cite{[GKT],[LS],[Lin],[Seregin1],[TX],[WW]}.
\begin{lemma}\label{presure}
For $0<\mu\leq\f{1}{2}\rho$ and $1<q<\infty$, there is an absolute constant $C$ independent of $\mu$ and $\rho$ such that
\begin{align}
P_{q}(\Pi,\mu)\leq & C\left(\f{\rho}{\mu}\right)^{5-2q}E_{2q}(u, \rho)+
C\left(\f{\mu}{\rho}\right)^{2q-2}P_{q}(\Pi,\rho).\label{lqlp2p}
\end{align}
\end{lemma}
\begin{proof}
We consider the usual cut-off function $\phi\in C^{\infty}_{0}(B(\rho))$ such that $\phi\equiv1$ on $B(\f{3}{4}\rho)$ with $0\leq\phi\leq1$ and
$|\nabla\phi |\leq C\rho^{-1},~|\nabla^{2}\phi |\leq
C\rho^{-2}.$

Due to the pressure equation $\partial_{i}\partial_{i}\Pi=\partial_{i}\partial_{j}(u_{j}  u_{i} )$, we may write
$$
\partial_{i}\partial_{i}(\Pi\phi)=-\phi \partial_{i}\partial_{j}( u_{j}  u_{i} )
+2\partial_{i}\phi\partial_{i}\Pi+\Pi\partial_{i}\partial_{i}\phi
,$$
from which it follows that for $x\in B(\f{3}{4}\rho)$,
\be\ba\label{pp}
\Pi(x)=&\Gamma \ast \{-\phi \partial_{i}\partial_{j}( u_{j}  u_{i} )
+2\partial_{i}\phi\partial_{i}\Pi+\Pi\partial_{i}\partial_{i}\phi
\}\\
=&-\partial_{i}\partial_{j}\Gamma \ast (\phi u_{j}  u_{i})\\
&+2\partial_{i}\Gamma \ast(\partial_{j}\phi u_{j}  u_{i})-\Gamma \ast
(\partial_{i}\partial_{j}\phi u_{j}  u_{i})\\
& -2\partial_{i}\Gamma \ast(\partial_{i}\phi \Pi) -\Gamma \ast(\partial_{i}\partial_{i}\phi \Pi)\\
=: &P_{1}(x)+P_{2}(x)+P_{3}(x),
\ea\ee
where $\Gamma$ stands for the normalized fundamental solution of Laplace's
equation in $\mathbb{R}^{3}$.

Since $\phi(x)=1 $  when $x\in B(\mu)$, we deduce that
\[
\Delta(P_{2}(x)+P_{3}(x))=0, x\in B(\mu).
\]
By the mean value property of harmonic functions, we have
\be\label{lem2.4.1}\ba
\int_{B(\mu)}|P_{2}(x)+P_{3}(x)|^{q}dx&\leq C\Big(\f{\mu}{\rho}\Big)^{3}\int_{B(\rho/2)}|P_{2}(x)+P_{3}(x)|^{q}dx\\
&\leq C\Big(\f{\mu}{\rho}\Big)^{3}\int_{B(\rho/2)}|\Pi(x)-P_{1}(x)|^{q}dx\\
&\leq C\Big(\f{\mu}{\rho}\Big)^{3}\int_{B(\rho/2)}|\Pi(x)|^{q}dx+
C\int_{B(\rho/2)}|P_{1}(x)|^{q}dx,\\
\ea\ee
where $C$ is independent of $\mu$ and $\rho.$

According to the classical Calder\'on-Zygmund Theorem, it is easy to get
\be\label{lem2.4.3}\ba
\int_{B(\mu)}|P_{1}(x)|^{q}dx&\leq \int_{B(\rho/2)}|P_{1}(x)|^{q}dx \leq C \int_{B(\rho)}|u |^{2q}dx. \\
 \ea\ee
Collecting   \eqref{pp}-\eqref{lem2.4.3} yields
\be\label{lem2.3}
\ba
\int_{B(\mu)}|\Pi|^{q}dx\leq& C\int_{B(\mu)}
|P_{1}|^{q}+(|P_{2}+P_{3}|)^{q}dx \\
\leq& C
\Big(\int_{B(\rho)}|u|^{2q}dx\Big)
  +C\left(\f{\mu}{\rho}\right)^{3}\int_{B(\rho)}|\Pi|^{q}dx.
\ea
\ee
Integrating this inequality in time gives
$$
\iint_{Q(\mu)}|\Pi|^{q}dxds\leq C
\Big(\iint_{Q(\rho)}|u|^{2q}dxds\Big)
  +C\left(\f{\mu}{\rho}\right)^{3}\iint_{Q(\rho)}|\Pi|^{q}dxds,
$$
which in turn implies the desired estimate \eqref{lqlp2p}.
\end{proof}

The following lemma can be regarded as the localized version of the Biot-Savart law.
\begin{lemma}\label{lemma2.4}
Let $(u,\Pi)$ be a suitable weak solution to (\ref{NS}) in $Q(1)$. If $\omega\in L^{p}L^{q}(Q(1))$ and
\be\label{11} \ba
\limsup_{r\to 0}\,\, r^{2-  \frac 3p -\frac 2q  }
\Big(\int^{0}_{-r^{2}}\Big(\int_{B(r)}|\omega|^{p}
dx\Big)^{\f{q}{p}}ds\Big)^{\f{1}{q}}
 \leq M,
\ea\ee
with $2\leq3/p+3/q<3$, $1<p,\,q$,
 then there   exist $r^{\ast}\leq1/2$ and $p',\,q'$  satisfying
$2\leq 3/p' +2/q' <3$ such that
\be\label{111}
E_{\ast,p',q'}(u,r)\leq C(M),
\ee for $0<r\leq r^{\ast}$.
\end{lemma}
\begin{proof}
By H\"older's  inequality, for any $2\leq\f{3}{p}+\f{2}{q}<3$,   we can choose $p'\leq p$, $ q'\leq q$ and $p'<3$ such that
$2\leq\f{3}{p'}+\f{2}{q'}<3$ and
$$W_{p',q'}(\omega,r )\leq CW_{p,q}(\omega,r ).$$
Using usual cut-off function in equations
  $\Delta u=-\text{curl} \,\omega $ together with the fact that $\text{div}\,u=0$,
  and the Calder\'on-Zygmund Theorem, we can deduce that
  \be\label{2.41}
  E_{\ast,p',q'}(u,\mu)\leq C\Big(\f{\rho}{\mu}\Big)W_{p',q'}(\omega,\rho) +C\Big(\f{\mu}{\rho}\Big)^{\f{3}{p'}-1}E_{\ast,p',q'}(u,\rho),
  \ee
where $0<\mu<\f{1}{2}\rho.$
This process is similar to the derivation of \eqref{lqlp2p} in
 Lemma \ref{presure}. For details, see Lemma 3.6 in \cite{[GKT]} or Lemma 3.6 in \cite{[TX]}.

Applying the classical elliptic  estimate to $\Delta u=-\text{curl}\,\omega $, we infer that $\nabla u\in L^{p'}L^{q'}(Q(1/2))$ and
$$
\|\nabla u\|_{L^{p'}L^{q'}(Q(1/2))}\leq C\|\omega\|_{L^{p'}L^{q'}(Q(1))}+C\|u\|_{L^{2}L^{\infty}(Q(1))}.
$$
With the help of this fact  and   the hypothesis \eqref{11}, the desired estimate \eqref{111}
can be derived  by iterating    \eqref{2.41}. Similar iteration technique will be  used in the proof of Theorem 1.3 in Section 4,
see also \cite{[CKN],[TX],[WW]}.
\end{proof}

The last lemma is concerned with  a  inequality which is a local version of
Sobolev and Ladyzhenskaya's inequalities \eqref{ls} and will play an important role in the proof of Theorem \ref{the3}.
\begin{lemma}\label{keylemma}
For $0<\sqrt{3}\mu\leq\rho$, there exists an absolute constant C independent of $\mu$ and $\rho$ such that
$$\ba
E_{10/3}(u,\mu)\leq& C\Big(\frac{\rho}{\mu}\Big)^{5/3} E^{2/3}(u,\rho)E^{1/3}_{\ast,3}(u,\rho)E^{2/3}_{\ast,h}(u,\rho)\\
&+C\Big(\frac{\rho}{\mu}\Big)^{5/3}
E(u,\rho)E^{2/3}_{\ast,h}(u,\rho)+C
\Big(\frac{\mu}{\rho}\Big)^{1/3}E_{10/3}(u,\rho).\ea$$
\end{lemma}
\begin{proof}Set
$$\overline{ u^{h}_{\sqrt{2}r} }
=\frac{1}{2\pi r^{2}}\iint_{x_{1}^{2}+x_{2}^{2}<2r^{2}}udx_{1}dx_{2},
 $$ where $r\geq\mu.$

It is obvious that
\begin{equation}\label{AL01}
\int_{B(\mu)}|u|^{\frac{10}{3}}dx\leq C\int_{B(\mu)}|u-\overline{ u^{h}_{\sqrt{2}r} }|^{\frac{10}{3}}dx+
C\int_{B(\mu)}|\overline{ u^{h}_{\sqrt{2}r} }|^{\frac{10}{3}}dx.
\end{equation}
On one hand, utilizing
  H\"older's inequality, we find that
\begin{equation}\label{AL02}
\ba
\int_{B(\mu)}|\overline{ u^{h}_{\sqrt{2}r} }|^{\frac{10}{3}}dx
\leq& C\int_{|x_{3}|<\mu}dx_{3}\iint_{|x_{1}|,|x_{2}|<\mu}d x_{1}d x_{2}  \Big|\frac{1}{r^{2}}
\iint_{x_{1}^{2}+x_{2}^{2}<2r^{2}}udx_{1}dx_{2}\Big|^{\frac{10}{3}}\\
\leq& C \Big(\frac{\mu}{r}\Big)^{2}
\int_{B(\sqrt{3}r)}|u|^{\frac{10}{3}}dx.
\ea
\end{equation}
On the other hand,   by H\"older's inequality, we have
\begin{equation}\label{AL03}
\ba
&\int_{B(\mu)}|u-\overline{ u^{h}_{\sqrt{2}r} }|^{\frac{10}{3}}dx\\
\leq&\iiint_{|x_{1}|,|x_{2}|,|x_{3}|<\mu}|u-\overline{ u^{h}_{\sqrt{2}r} }|^{\frac{10}{3}}dx\\
\leq& \Big\||u-\overline{ u^{h}_{\sqrt{2}r} }|^{\frac{4}{3}}\Big\|_{L^{\f{5}{2}}(|x_{1}|,|x_{2}|,|x_{3}|<\mu)}
\Big\|\|u-\overline{ u^{h}_{\sqrt{2}r} }\|_{L^{10}(|x_{3}|<\mu)}   \Big\|_{L^{2}(|x_{1}|,|x_{2}| <\mu)}\\&\times
\Big\|\|u-\overline{ u^{h}_{\sqrt{2}r} }\|_{L^{2}(|x_{3}|<\mu)}   \Big\|_{L^{10}(|x_{1}|,|x_{2}| <\mu)}\\
\leq&
\Big\|\|u-\overline{ u^{h}_{\sqrt{2}r} }\|_{L^{10}(|x_{3}|<\mu)}   \Big\|_{L^{2}(|x_{1}|,|x_{2}| <\mu)} ^{\f{5}{3}}
\Big\|\|u-\overline{ u^{h}_{\sqrt{2}r} }\|_{L^{2}(|x_{3}|<\mu)}   \Big\|_{L^{10}(|x_{1}|,|x_{2}| <\mu)}^{\f{5}{3}}.
\ea
\end{equation}
In the light of  the Gagliardo-Nirenberg inequality \cite{[Nirenberg]}
  on the bounded domain in $\mathbb{R}$,
  we know that
$$
\|f(x_{3})\|_{L^{10}(|x_{3}|<\mu)}
\leq C
\|f(x_{3})  \|_{L^{2}(|x_{3}|<\mu)}^{\frac{3}{5}}
\|\partial_{3}f(x_{3})\|_{L^{2}(|x_{3}|<\mu)}^{\frac{2}{5}}
+C\mu^{-\f{2}{5}}\|f(x_{3})\|_{L^{2}(|x_{3}|<\mu)}$$
for any $f\in W^{1,2}(-\mu,\mu)$.\\
With the help of this inequality and H\"older's inequality, we easily get
 \begin{align}
&\Big\|\|u-\overline{ u^{h}_{\sqrt{2}r} }\|_{L^{10}(|x_{3}|<\mu)}   \Big\|_{L^{2}(|x_{1}|,|x_{2}| <\mu)}\nonumber\\
\leq& C \Big\|\|u-\overline{ u^{h}_{\sqrt{2}r} }\|_{L^{2}(|x_{3}|<\mu)}^{\frac{3}{5}}
\|\partial_{3}u-\partial_{3}\overline{ u^{h}_{\sqrt{2}r} }
\|_{L^{2}(|x_{3}|<\mu)}^{\frac{2}{5}}  \Big\|_{L^{2}(|x_{1}|,|x_{2}|
<\mu)}\nonumber\\&+ C\mu^{-\f{2}{5}}  \Big\|\|u-\overline{ u^{h}_{\sqrt{2}r} }
\|_{L^{2}(|x_{3}|<\mu)}\Big\|_{L^{2}(|x_{1}|,|x_{2}| <\mu)}\nonumber\\
\leq& C  \|u-\overline{ u^{h}_{\sqrt{2}r} }\|_{L^{2}(|x_{1}|,|x_{2}|,|x_{3}|<\mu)}^{\frac{3}{5}}
\|\partial_{3}u-\partial_{3}\overline{ u^{h}_{\sqrt{2}r} }\|_{L^{2}(x_{1}|,|x_{2}|,|x_{3}|<\mu)}^{\frac{2}{5}}
\nonumber\\&+ C\mu^{-\f{2}{5}}   \|u-\overline{ u^{h}_{\sqrt{2}r} }\|_{L^{2}(|x_{1}|,|x_{2}|,| x_{3}|<\mu)} \nonumber\\
=& C  \Big\|\|u-\overline{ u^{h}_{\sqrt{2}r} }\|_{L^{2}(|x_{1}|,|x_{2}| <\mu)}\Big\|_{L^{2}(|x_{3}|<\mu)}^{\frac{3}{5}}
\|\partial_{3}u-\partial_{3}\overline{ u^{h}_{\sqrt{2}r} }\|_{L^{2}(|x_{1}|,|x_{2}| <\mu)}\Big\|^{\frac{2}{5}}_{L^{2}(|x_{3}|<\mu)}
\nonumber\\&+ C\mu^{-\f{2}{5}}   \Big\|\|u-\overline{ u^{h}_{\sqrt{2}r} }\|_{L^{2}(|x_{1}|,|x_{2}| <\mu)}\Big\|_{L^{2}(|x_{3}|<\mu)}
\nonumber\\
\leq& C  \Big\|\|u\|_{L^{2}( x_{1} ^{2}+ x_{2} ^{2} <2r^{2})}\Big\|_{L^{2}(|x_{3}|<\mu)}^{\frac{3}{5}}
\Big\|\|\partial_{3}u\|_{L^{2}( x_{1} ^{2}+ x_{2} ^{2} <2r^{2})}\Big\|^{\frac{2}{5}}_{L^{2}(|x_{3}|<\mu)}
\nonumber\\& + C\mu^{-\f{2}{5}}  \Big\|  \|u\|_{L^{2}( x_{1} ^{2}+ x_{2} ^{2} <2r^{2})}\Big\|_{L^{2}(|x_{3}|<\mu)}
\nonumber\\
\leq& C  \|u\|_{L^{2}(B(\sqrt{3}r))} ^{\frac{3}{5}}
\|\partial_{3}u\|_{L^{2}(B(\sqrt{3}r))} ^{\frac{2}{5}}
 + C\mu^{-\f{2}{5}}
  \|u\|_{L^{2}(B(\sqrt{3}r))}, \nonumber\end{align}
where we have used the fact that
 $\|u-\overline{ u^{h}_{\sqrt{2}r} }\|_{L^{2}(x_{1}^{2}+x_{2}^{2}<2r^{2})}\leq 2 \|u\|_{L^{2}( x_{1}^{2}+x_{2}^{2} <2r^{2})}$.\\
In addition, taking advantage of Minkowski's inequality  and Poinc\'are-Sobolev's inequality on the ball in $\mathbb{R}^{2}$, we find that
$$\ba
\Big\|\|u-\overline{ u^{h}_{\sqrt{2}r} }\|_{L^{2}(|x_{3}|<\mu)}   \Big\|_{L^{10}(|x_{1}|,|x_{2}| <\mu)}&\leq \Big\|\|u-\overline{ u^{h}_{\sqrt{2}r} }\|_{L^{10}( x_{1}^{2}+x_{2}^{2} <2r^{2})}\Big\|_{L^{2}(|x_{3}|<\mu)}\\
&\leq C\Big\|\|u \|^{\frac{1}{5}}_{L^{2}( x_{1}^{2}+x_{2}^{2} <2r^{2})}\|\nabla_{h}u \|^{\frac{4}{5}}_{L^{2}( x_{1}^{2}+x_{2}^{2} <2r^{2})}\Big\|_{L^{2}(|x_{3}|<\mu)}\\
&\leq  C\|u \|^{\frac{1}{5}}_{L^{2}(B(\sqrt{3}r))}\|\nabla_{h}u
 \|^{\frac{4}{5}}_{L^{2}(B(\sqrt{3}r))}.
\ea$$
Plugging  these two estimates into \eqref{AL03}, we infer that
$$\ba&\int_{B(\mu)}|u-\overline{ u^{h}_{\sqrt{2}r} }|^{\frac{10}{3}}dx\\
\leq& C \|u\|^{\f{4}{3}}_{L^{2}(B(\sqrt{3}r))}
\|\partial_{3}u\|^{\f{2}{3}}_{L^{2}(B(\sqrt{3}r))}
\|\nabla_{h}u\|^{\f{4}{3}}_{L^{2}(B(\sqrt{3}r))}
+C\mu^{-\f{2}{3}}\|u\|^{2}_{L^{2}(B(\sqrt{3}r))}\|\nabla_{h}u\|^{\f{4}{3}}_{L^{2}(B(\sqrt{3}r))},\ea$$
which together with \eqref{AL01} and \eqref{AL02} implies
$$\ba\int_{B(\mu)}|u|^{\frac{10}{3}}dx\leq& C \|u\|^{\f{4}{3}}_{L^{2}(B(\sqrt{3}r))}
\|\partial_{3}u\|^{\f{2}{3}}_{L^{2}(B(\sqrt{3}r))}
\|\nabla_{h}u\|^{\f{4}{3}}_{L^{2}(B(\sqrt{3}r))}
\\&+C\mu^{-\f{2}{3}}\|u\|^{2}_{L^{2}(B(\sqrt{3}r))}\|\nabla_{h}u\|^{\f{4}{3}}_{L^{2}(B(\sqrt{3}r))}
+C \Big(\frac{\mu}{r}\Big)^{2} \int_{B(\sqrt{3}r)}|u|^{\frac{10}{3}}dx. \ea$$
Integrating this inequality with respect to $t$
on $(-\mu^{2},0)$ and utilizing  H\"older's  inequality yields that
$$\ba
 &\iint_{Q(\mu)}|u|^{\f{10}{3}}dxds\\\leq&
C\left(\sup_{-r^{2}\leq s\leq0}\int_{B(\sqrt{3}r)}|u|^{2}dx\right)^{\f{2}{3}}
\left(\iint_{Q(\sqrt{3}r)}|\partial_{3}u|^{2}dxds\right)^{\f{1}{3}}
\left(\iint_{Q(\sqrt{3}r)}|\nabla_{h}u|^{2}dxds\right)^{\f{2}{3}}\\
&+C\left(\sup_{-r^{2}\leq s\leq0}\int_{B(\sqrt{3}r)}|u|^{2}dx\right)
\Big(\iint_{Q(\sqrt{3}r)}|\nabla_{h}u|^{2}dxds\Big)^{\f{2}{3}}
\\ &+C \Big(\frac{\mu}{r}\Big)^{2}
\iint_{Q(\sqrt{3}r)}|u|^{\f{10}{3}}dxds,
\ea$$
which in turn implies the desired estimate.
\end{proof}

\section{Proofs of Theorem \ref{the2} and \ref{the1}}

In the spirit of   \cite{[Seregin2]}, Theorem \ref{the2} and \ref{the1} turn
out to be the corollaries of the following propositions.
\begin{proposition}\label{prop2}
Assume that the pair $(u,\Pi)$ is a suitable weak solution to the Navier-Stokes equations. For any $M>0$, there exists a positive constant $\varepsilon (M)$ such that if
$$r^{2-\f{3}{p}-\f{2}{q}}\Big(\int^{0}_{-r^{2}}\Big(\int_{B(r)}|\nabla u|^{p }dx\Big)^{\f{q}{p }}ds\Big)^{\f{1}{q}}\leq M, $$
and
\be\label{propc2}r_{\ast}^{2-\f{3}{p}-\f{2}{q }}
\Big(\int^{0}_{-r_{\ast}^{2}}\Big(\int_{B(r_{\ast})}
|\omega_{h}|^{p}dx\Big)^{\f{q }{p}}ds\Big)^{\f{1}{q }}\leq \varepsilon,\ee
for some $ p,\,q  $
$\text{with}~ 2\leq 3/p+2/q<3,1<p,q\leq\infty$
and some $r_{\ast}\in \big(0, \min\{1/2,(E_{3}(u,1)+P_{3/2}(u,1))^{-2}\}\big]$,
then $(0,0)$ is regular point.
\end{proposition}
\begin{proof}
Assume that the statement fails,
 then there exists  a sequence ${(u^{k},\Pi^{k})}$ of the suitable weak solutions to Navier-Stokes equations such that
\be\label{p1}
E_{\ast,p,q}(u^{k},r)\leq M,
\ee
for all $0<r\leq1$ and
\be\label{bl1}
W_{p ,q }(\omega_{h}^{k},r_{k})\leq \f{1}{k},
\ee
for $r_{k}\in \big(0, \min\{1/2,(E_{3}(u,1)+P_{3/2}(u,1))^{-2}\}\big]$ and  $(0,\,0)$ is a singular point of $u^{k}$. Therefore, in light of Theorem \ref{Lin}, there exists an absolute constant $\varepsilon_{0}>0$ such that
$$
E_{3}( u^{k},r)+P_{3/2}(\Pi^{k},r)\geq\varepsilon_{0},
$$
for any $0<r\leq r_{k}.$
With the help of Lemma \ref{lem1} and \eqref{p1}, we can deduce that
$$
E(u^{k},r)+E_{\ast}(u^{k},r)+P_{3/2}(\Pi^{k},r)\leq C(M),
$$
for all $0<r\leq r_{k}$.

Set $v^{k}(x,t)=r_{k}u^{k}(r_{k}x,r^{2}_{k}t),
~q^{k}(x,t)=r_{k}^{2}\Pi^{k}(r_{k}x,r^{2}_{k}t),\,
\widetilde{ \omega}^{k}(x,t)=r_{k}^{2}\omega^{k}(r_{k}x,r^{2}_{k}t)$, where $\omega^{k}=\text{curl} \,u^{k}$.
For any $R>0$, a straightforward computation gives
\begin{align}
&E (v^{k}, R)=E (u^{k},r_{k}R),~
&E&_{3}(v^{k}, R)=E_{3}(u^{k},r_{k}R),\nonumber\\ &P_{3/2}(q^{k},R)=P_{3/2}(\Pi^{k},r_{k}R),
~
&W& _{p,q}(\widetilde{\omega}_{h}^{k},r)=
W_{p,q}(\omega_{h}^{k},r_{k}r). \nonumber
 \end{align}
 This implies that, for all $0<r\leq1$ and $k\in \mathbb{N}$,
\begin{align}
&E(v^{k},r)+E_{3}(v^{k},r)+P_{3/2}(q^{k},r)\leq C(M),\label{irr2}\\
&W _{p,q}(\widetilde{\omega}_{h}^{k},1)\leq\f{1}{k},\label{irr3}\\
&E_{3}( v^{k},r)+P_{3/2}(q^{k},r)\geq\varepsilon_{0}.\label{irr1}
 \end{align}
It is obvious
 to see that the pair $(v^{k},~q^{k})$ solves the following system
 \be
\left\{\ba
&v^{k}_{t}+v^{k}\cdot\nabla v^{k}-\Delta v^{k}+\nabla q^{k}=0,\\
&\te{div}~ v^{k}=0,
\ea\right.
\ee
in the sense of  distribution on $Q(1)$.
For any text function $\phi\in L^{3}(-1,0; W_{0}^{2,2}(B(1))),$
simple computations give that
$$
\ba
\iint_{Q(1)} (v^{k}\otimes v^{k}):\nabla \phi dxdt&\leq \|v^{k}\|_{L^{\infty}(L^{2})}\|v^{k}\|_{L^{2}(L^{6})}
\|\nabla \phi\|_{L^{2}(L^{3})}\leq C(M) \|\phi\|_{L^{3}( W_{0}^{2,2})},\\
\iint_{Q(1)}|\nabla  v^{k}||\nabla \phi| dxdt&\leq \|\nabla  v^{k}\|_{L^{2}(L^{2})}\|\nabla  \phi\|_{L^{2}(L^{2})}\leq C(M) \|\phi\|_{L^{3}( W_{0}^{2,2})},\\
\iint_{Q(1)}q^{k}\text{div}\phi dxdt&\leq \| q^{k}\|_{L^{3/2}(L^{3/2})}\|  \phi\|_{L^{3}(L^{3})}\leq C(M) \|\phi\|_{L^{3}( W_{0}^{2,2})},
\ea
$$
which implies
 $\partial_{t}v^{k}\in L^{\f{3}{2}}(-1,0;(W_{0}^{2,2}(B(1)))^{\ast})$, where
 $(W_{0}^{2,2}(B(1)))^{\ast}$ denotes the dual space of $W_{0}^{2,2}(B(1))$. This together with the fact that $\nabla  v^{k}\in L^{2}L^{2}(Q(1))$
allows us to obtain a subsequence (still denoted by $k$) by using Aubin-Lions Lemma (\cite[Theorem 2.1, p.184]{[Teman2]}) such that
$$v^{k}\rightarrow  v~~~~\text{in}~~ L^{2}(Q(1)).$$
By means of  the interpolation inequality, it follows from $v^{k}\in L^{\infty}(-1,0;L^{2}(B(1)))\cap L^{2}(-1,0; W^{1,2} (B(1)))$ that
$v^{k}\in L^{10/3}(Q(1))$.
Thus, by a diagonalization process, we find that
\begin{align}
&v^{k}\rightarrow v \quad \text{in}\quad L^{3}(Q(1)),\label{0}\\
&\omega_{h}^{k}\rightharpoonup0 \quad \text{in}\quad L^{p}L^{q}(Q(1)),\label{propc11}\\
&q^{k}\rightharpoonup q \quad \text{in}\quad  L^{3/2}(Q(1)).\label{00}\end{align}
Furthermore,
 $ v=(v_{1},v_{2},v_{3})$ and $q $ solve
 \be\label{spcial1}
\left\{\ba
&v_{t}+v  \cdot\nabla v -\Delta v +\nabla q =0,\\~~&\te{div}~ v =0,\,
\omega_{h} =0,
\ea\right.
\ee
   in the sense of distribution.
Form   Lemma \ref{serr1} in Appendix or
   Chae, Kang and Lee's main result in \cite{[CKL]},
we know that there exist a constant $0<r'\leq1$ such that
    \be\label{m} |v(x,t)| \leq C(M),~~~~~~~ (x,t)\in Q(r').\ee
Notice that \eqref{irr1} is also true for the convergence
subsequence obtained in the \eqref{0}-\eqref{00}.
 Based on this, we can pass
to the limit in \eqref{irr1} to arrive at
\be\label{irr}E_{3}( v ,r)+\lim_{{k}\rightarrow\infty}P_{3/2}(q^{{k}},r)\geq\varepsilon_{0}.\ee
Employing Lemma \ref{presure} for $q=3/2$ and any $0<2r<\sqrt{r}\leq \min\{1/4,r' \}$, we can get
$$
P_{3/2}(q^{n_{k}},r) \leq  C\left(\f{\sqrt{r}}{r}\right)^{2}E_{3}(v^{{k}}, \sqrt{r})+
C\left(\f{r}{\sqrt{r}}\right)P_{3/2}(q^{{k}},\sqrt{r}) \leq C\sqrt{r}, $$
where we have used \eqref{irr2} and \eqref{m}.
This estimate together with \eqref{irr}  yields that
$$Cr^{3}+C\sqrt{r}\geq\varepsilon_{0},$$
for all $0<r<\min\{\f{1}{8},\, \f{1}{2}r'\}$.
 This will lead to a contradiction when $r$ is sufficiently small, which concludes
 the proof of this proposition.
\end{proof}
Neustupa and  Penel's result in \cite{[NP]} ensures the  interior
regularity of system \eqref{spe1}.
With the help of this fact together with Lemma \ref{lem1}, it is not hard to show the following proposition by modifying slightly the above proof. We omit the detail here.
 \begin{proposition}\label{prop1}
Assume that the pair $(u,\Pi)$ is a suitable weak solution to the Navier-Stokes equations. For any $M>0$, there exists a positive constant $\varepsilon (M)$ such that if
$$\sup_{0<r\leq1}r^{1-\f{3}{p}-\f{2}{q}}\Big(\int^{0}_{-r^{2}}
\Big(\int_{B(r)}|u|^{p}dx\Big)^{\f{q}{p}}ds\Big)^{\f{1}{q}}\leq M, $$
and
\be\label{propc1}r_{\ast}^{1-\f{3}{p}-\f{2}{q}}
\Big(\int^{0}_{-r_{\ast}^{2}}
\Big(\int_{B(r_{\ast})}|u_{3}|^{p}dx\Big)
^{\f{q}{p}}ds\Big)^{\f{1}{q}}\leq \varepsilon,\ee
for  $ p,\,q $   satisfying
$ 1\leq 3/p+2/q<2,1<p,\,q\leq\infty,
$
and some $r_{\ast}\in \big(0, \min\{\frac{1}{2},(E_{3}(u,1)+P_{3/2}(u,1))^{-2}\}\big]$,
then $(0,0)$ is regular point.
\end{proposition}
Now we are in a position to complete the proofs of Theorem \ref{the2}
 and \ref{the1}.
\begin{proof}[Proof of Theorem \ref{the2}]
It follows from \eqref{the2c} and Lemma \ref{lemma2.4} that there is a constant $r_{\ast}$ such that
$$E_{\ast,p',q'}(u,r)\leq C(M),\,\forall\, 0<r\leq r_{\ast} .$$
Set $v(x,t)=r_{\ast}u(r_{\ast}x,r_{\ast}^{2}t),
q(x,t)=r^{2}_{\ast}p(r_{\ast}x,r_{\ast}^{2}t)$ and $\widetilde{\omega}= r^{2}_{\ast}\omega(r_{\ast}x,r_{\ast}^{2}t),$ where $\omega=\text{curl}\,u$.
Therefore, for any $0<r\leq1$,
$E_{\ast,p',q'}(v,r)= E_{\ast,p',q'}(u,r r_{\ast})\leq C(M).$
 Moreover, we also have
$$\liminf\limits_{r\rightarrow0} W_{p,q }(\widetilde{\omega},r)=\liminf\limits_{r\rightarrow0} W_{p,q }(\omega,rr_{\ast})=\liminf\limits_{rr_{\ast}\rightarrow0} W_{p,q }(\omega,rr_{\ast})<\varepsilon_{11}.$$
From this inequalities
 and the definition of the limit inferior,
  we deduce that the suitable weak solution ($v,\, q$)
  satisfies the condition of Proposition \ref{prop2}.
 Hence, $(0,0)$ is a regular point of $v$, which means that $(0,0)$ is also regular point of $u$. Therefore, Theorem \ref{the2} is proved.
\end{proof}

\begin{proof}[Proof of Theorem \ref{the1}]
(1)\,
By the condition \eqref{the1c},
 we find that there exists a constant $r_{\ast} >0$ such that
$$
\sup_{0<r\leq r_{\ast}} E_{p,q}(u,r)\leq M.
$$
Set $v(x,t)=r_{\ast}u(r_{\ast}x,r_{\ast}^{2}t), \, q(x,t)=r^{2}_{\ast}p(r_{\ast}x,r_{\ast}^{2}t)$.
 It is clear that, for any $0<r\leq 1$,
$$E_{p,q}(v,r)=E_{p,q}(u,rr_{\ast})\leq M,$$ and $$\liminf\limits_{r\rightarrow0} E_{p,q}(v,r)=\liminf\limits_{r\rightarrow0} E_{p,q}(u,rr_{\ast})=\liminf\limits_{rr_{\ast}\rightarrow0} E_{p,q}(u,rr_{\ast})<\varepsilon_{21}.$$
Now we can employ  Proposition \ref{prop1} to obtain that $(0,0)$ is a regular point of $v$, which in turn implies that $(0,0)$ is also a regular point of $u$. Therefore, we finish the proof of the first part of this theorem.

\noindent(2) \,
By    H\"older's inequality, without loss of generality, we just deal with the case $3/p+2/q=3$.
Thanks to our assumption  \eqref{the1asum} and Lemma \ref{lem1}, we see that
\be\label{1.2proof1}E(u,r)+E_{\ast}(u,r)+P_{3/2}(p,r)\leq C(M),\ee
for any 0$<r\leq\min\{\frac{1}{4},(E_{3}(u,1)+P_{3/2}(u,1))^{-2}\},$ which yields that $E_{3}(u,r)\leq C(M)$.

Using the triangle inequality,  H\"older's inequality, Sobolev-Poinc\'are's
inequality and Gagliardo-Nirenberg's inequality, one can deduce that
\be\label{1.2proof2}\ba
E_{3}(u_{3},\mu)&\leq C  E_{3}(u_{3}-{\overline{u_{3}}}_{r''},\mu) +
    C\Big(\f{\mu}{r''}\Big)E_{3}(u_{3},r'')\\
&\leq C \Big(\f{r''}{\mu}\Big)^{2}E^{\f{1}{q}}(u_{3},r'')
E^{1-\f{1}{q}}_{\ast}(u_{3},r'')
E_{\ast;p,q}(u_{3},r'')+
     C\Big(\f{\mu}{r''}\Big)E_{3}(u_{3},r''),\\
\ea\ee
where $0<\mu\leq r''$ and ${\overline{u_{3}}}_{r''}=\f{1}{|B(r'')|}\int_{B(r'')}u_{3}dx$.
 For  detailed computation, \vspace{0.09cm}we refer the reader to
\cite[Lemma 3.2, 3.3 and 3.5, p.168-171]{[GKT]}.

It follows from \eqref{1.2proof1} and \eqref{1.2proof2} that
$$E_{3}(u_{3},\mu)\leq C(M)\Big(\f{r''}{\mu}\Big)^{2}\varepsilon_{22}+ C\Big(\f{\mu}{r''}\Big),$$
where $r''\leq\min\{\frac{1}{4},(E_{3}(u,1)+P_{3/2}(u,1))^{-2}\}.$
Fix $r''$, then choose $\mu'$ and $\varepsilon$ such that $E_{3}(u_{3},r)\leq \varepsilon,$ for $r\leq\mu'$.
Now we can make use of the result in the first part of Theorem \ref{the1} to finish the proof of the second part.
\end{proof}

\section{Proof of Theorem \ref{the3}}

\begin{proof}[Proof of Theorem \ref{the3} ]
The condition  \eqref{the3c} guarantees that there is a constant $r_{1}$ such that
$$
E_{\ast,h}(u,r)\leq \varepsilon_{31},~\text{for any} ~~~~~ 0<r\leq r_{1}.
$$
According to the regularization condition \eqref{tx} (or condition \eqref{tsai1}), it is enough
to show that there exists a constant $0<r_{2}\leq r_{1}$ such that
$$
E_{10/3}(u,r)\leq \varepsilon_{2},~\text{for any} ~~~~~ 0<r\leq r_{2}.
$$
By the fundamental
 property of the usual nonnegative cut-off function,
it follows from the local energy inequality \eqref{loc} that
$$
E(u,r)+E_{\ast}(u,r) \leq  C\Big[E_{2}(u,2r)+E_{3}(u,2r)
+E^{1/3}_{3}(u,2r)P^{2/3}_{3/2}(p,2r)\Big],
$$
which in turn implies that
$$\ba
&\varepsilon_{31}^{1/9}E(u,r)+\varepsilon_{31}^{1/9}E_{\ast}(u,r)\\
\leq& C \varepsilon_{31}^{1/9}\Big(E_{10/3}(u,2r)\Big)^{3/5}+
 C \varepsilon_{31}^{1/9}\Big(E_{10/3}(u,2r)\Big)^{9/10}\\
 &+ C \varepsilon_{31}^{1/9}
 \Big(E_{3/10}(u,2r)\Big)^{10/3}\Big(P_{5/3}(p,2r)\Big)^{3/5}\\
 \leq& C E_{10/3}(u,2r)+C\varepsilon_{31}^{5/18}+C\varepsilon_{31}+
 \varepsilon_{31}^{10/63}
 \Big(P_{5/3}(p,2r)\Big)^{6/7}\\
 \leq& C_{1} E_{10/3}(u,2r)+\varepsilon_{31}^{1/18 }\Big(\frac{\rho}{r}\Big)^{5/3}\varepsilon^{1/9}
P_{5/3}(p,\rho) +C_{3}\varepsilon_{31}^{1/9},
\ea$$
where we have used  H\"older's and  Young's inequalities and the fact that
$$P_{5/3}(p,2r)\leq\Big(\frac{\rho}{r}\Big)^{5/3}P_{5/3}(p,\rho)~\text{for} ~~2r\leq\rho\leq r_{1}.$$
By Lemma \ref{presure} with $q=5/3$, we get
$$
\varepsilon_{31}^{1/9}
P_{5/3}(p,2r) \leq\varepsilon_{31}^{1/9}C \Big(\frac{\rho}{r}\Big)^{5/3}E_{10/3}(\rho)
+C\Big(\frac{r}{\rho}\Big)^{4/3}
\varepsilon_{31}^{1/9}P_{5/3}(\rho).$$

Set
$$
\psi(r)=E_{10/3}(u,r)+\varepsilon_{31}^{1/9}E(u,r)+\varepsilon_{31}^{1/9}E_{\ast}(u,r)
 +\varepsilon_{31}^{1/9}
P_{5/3}(p,r).$$
 Thanks to Lemma \ref{keylemma}, we arrive at
$$
E_{10/3}(u,r)\leq C
\Big(\frac{\rho}{\mu}\Big)^{5/3} \psi(\rho)\varepsilon_{31}^{5/9}
+C
\Big(\frac{\mu}{\rho}\Big)^{1/3}E_{10/3}(u,\rho),$$
which in turn implies that
\be\label{120}\ba
 \psi(r)
\leq& C\Big(\frac{\rho}{r}\Big)^{5/3} \psi(\rho)\varepsilon_{31}^{5/9}
+C \Big(\frac{r}{\rho}\Big)^{1/3} E_{10/3}(u,\rho) +
\varepsilon_{31}^{1/18}\Big(\frac{\rho}{r}\Big)^{5/3}\varepsilon_{31}^{1/9}
P_{5/3}(p,\rho) +C_{3}\varepsilon_{31}^{1/9}\\
&+\varepsilon_{31}^{1/9}C
\Big(\frac{\rho}{r}\Big)^{2}E_{10/3}(u,\rho)+C\Big(\frac{r}{\rho}\Big)
\varepsilon_{31}^{1/9}P_{5/3}(p,\rho)\\
\leq& C \varepsilon_{31}^{1/18}\Big(\frac{\rho}{r}\Big)^{5/3}\psi(\rho)+
 C \Big(\frac{r}{\rho}\Big)^{1/3}\psi(\rho)+\varepsilon_{31}^{1/9}C\\
\leq& C_{1} \lambda^{-5/3} \varepsilon_{31}^{1/18} \psi(\rho)+
 C_{2}\lambda^{1/3} \psi(\rho)+\varepsilon_{31}^{1/9}C_{3},
 \ea\ee
where $\lambda= \mu/\rho \leq 1/4$.
Choosing $\lambda,\, \varepsilon_{31}$ such that
 $$q=2C_{2}\lambda^{1/3}<1\quad \text{and}\quad
 \varepsilon_{31}=\min\left\{\Big(\f{q\lambda^{5/3}}{2C_{1}}\Big)^{18},
 \Big(\f{(1-q)\lambda^{10/3}\varepsilon_{2}}{2C_{3}}\Big)^{9}\right\}.$$
It follows from
\eqref{120} that
$$
\psi(\lambda\rho)\leq q\psi(\rho)+\varepsilon_{31}^{1/9}C_{3}.
$$
Iterating this inequality, we see that
$$
\psi(\lambda^{k}\rho)\leq q^{k}\psi( \rho)+\f{1}{2}\lambda^{5/3}\varepsilon_{2}.
$$
By the definition of $\psi(r)$, we know that there exists a positive number $K_{0}$ such that
$$
q^{K_{0}}\psi(r_{1})\leq  4\f{C((\|u\|_{L^{\infty}L^{2}},\|u\|_{L^{2}W^{1,2}},
\|p\|_{L^{5/3}L^{5/3}}))}{r_{1}^{5/3}}q^{K_{0}}
\leq\f{1}{2}\lambda^{5/3}\varepsilon_{2}.
$$
Let $r_{2}=\lambda^{K_{0}}r_{1}$. Therefore, for all
 $0<r\leq r_{2},$ there exists a constant $k\geq K_{0}$
 such that $\lambda^{k+1}r_{1}\leq r\leq\lambda^{k}r_{1}$.
 Straightforward calculations show that
$$\ba
E_{10/3}(u,r)&=\f{1}{r^{5/3}}\iint_{Q(r)}|u|^{10/3}dxdt\\
&\leq\f{1}{\lambda^{k+1}r_{1}}\iint_{Q(\lambda^{k}r_{1})}|u|^{10/3}dxdt\\
&\leq\f{1}{\lambda^{5/3}}\psi(\lambda^{k}r_{1})\\
&\leq\f{1}{\lambda^{5/3}}(q^{k-K_{0}}q^{K_{0}}\psi(r_{1})
+\f{1}{2}\lambda^{5/3}\varepsilon_{2})\\
&\leq \varepsilon_{2},
\ea$$
which ends the proof of this theorem.
\end{proof}

\appendix
\section{Appendix}
\label{appendix}

\begin{lemma}\label{serr0}
Let $k(x,t)$ denote the fundamental solution of heat equation with $k(x,t)=0$ for $t\leq 0$,
and $f\in L^{q'}(0,T; L^{q}(\Omega))$, and set
$$
k\ast f(x,t)=\int_{0}^{T}\int_{\Omega}k(x-y,t-s)f(y,s)dyds.
$$
Then
$$
\|k\ast f\|_{L^{r'}(0,T;L^{r}(\Omega))}\leq C\| f\|_{L^{q'}(0,T;L^{q}(\Omega))},$$
with $$n\Big(\f{1}{q}-\f{1}{r}\Big)+2\Big(\f{1}{q'}-\f{1}{r'}\Big)<2.$$
\end{lemma}

\begin{proof}
  Since the proof is standard, we omit it (see, for example, \cite{[Serrin]}).
\end{proof}

\begin{lemma}\label{serr1}Let the pair $(v,\pi)$ be the weak solution to the following system
$$\left\{\ba
&v_{t}-\Delta v+v\cdot\nabla v+\nabla \pi=0,\\
&\mathrm{div}\, v=0,~\omega_{h}=0,
\ea\right.$$
where $\omega=$\text{curl} $v$.
Then one has $v\in L^{\infty}_{loc}$.
\end{lemma}
\begin{proof}
The proof relies on a bootstrapping argument.
Since  $\omega_{h}=0$ and div\,$\omega=0$,  the equations of vorticity
 read as
 $$
{\omega_{3}}_{t}-\Delta \omega_{3}+v\cdot\nabla \omega_{3}=\omega_{3}\partial_{3} v_{3},~\partial_{3}\omega_{3}=0.
$$
By means of $\partial_{3}\omega_{3}=0$, we can control the stretching term
$\omega_{3}\partial_{3} v_{3}$ just as the term $v\cdot\nabla \omega_{3}$
when we apply the standard energy method to get $\omega_{3}\in L^{\infty}L^{2}\cap L^{2}(W^{1,2})$ under the condition $v\in L^{\infty}L^{2}\cap L^{2}(W^{1,2}). $

Thanks to $v\cdot\nabla v=\f{1}{2}\nabla v^{2}+\omega\times v,$
 the original system can be rewritten as
$$\left\{\ba
&v_{t}-\Delta v+\nabla \Pi=-\omega\times v=-(-v_{2}\omega_{3},v_{1}\omega_{3},0),  \\
&\text{div}\, v=0,~~\Pi=  \pi+  \f{1}{2}v^{2}.
\ea\right.$$
Notice that  if we assume that $u\in L^{p_{k}}_{t,x},\,k=1,2,\cdots$, then $\omega\times v\in L^{\f{10p_{k}}{3p_{k}+10}}.$
It follows from the classical interior estimate  for the Stokes system that
$\nabla \Pi\in L^{\f{10p_{k}}{3p_{k}+10}}.$
Therefore
$v_{t}-\Delta v=-\omega\times v-\nabla \Pi\in L^{\f{10p_{k}}{3p_{k}+10}}. $
Notice also that $v=k\ast(-\omega\times v-\nabla \Pi)+H(x,t)=:v^{1}+H(x,t)$, where
$H(x,t)$ is the solution of the heat equations. At first, it is clear that $H(x,t)\in L^{\infty}$. Secondly, by Lemma \ref{serr0}, we can get $v^{1} \in L^{p_{k+1}}$ where $p_{k+1}$ satisfies
$$
3\Big(\f{3}{10}+\f{1}{p_{k}}-\f{1}{p_{k+1}}\Big)
+2\Big(\f{3}{10}+\f{1}{p_{k}}-\f{1}{p_{k+1}}\Big)<2,
$$
namely,
$$
 \f{1}{p_{k}}-\f{1}{p_{k+1}}<\f{1}{10}.
$$
Set $
 \f{1}{p_{k}}-\f{1}{p_{k+1}}=\f{1}{20},
$ and recall that $p_{1}=10/3$.
Then, after a finite number of bootstrapping steps,
we finally obtain $v\in L^{\infty}_{t,x}$.
\end{proof}

\begin{remark}
After this paper was submitted for publication, we learnt that the regular condition \eqref{wz1} was improved by Wang, Zhang and Zhang 
 in the preprint \cite{[WZZ]}. We would like to point out that our Theorem   \ref{the1} is different from  Theorem   1.2 in \cite{[WZZ]}.  Especially, the fist result of Theorem   \ref{the1}
    does not seem to be comparable to that of \cite{[WZZ]}.
 The first author  express his thank
to Professor Zhifei Zhang for  providing the paper \cite{[WZ2]}.
\end{remark}

\noindent
{\bf Acknowledgements:}
The second author is supported in part by the National Natural Science Foundation of China under grant No.11101405 and the President Fund of UCAS.


\end{document}